\documentclass[12pt]{amsart}
\font\emailfont=cmtt10

\usepackage{amsmath,amsthm,amsfonts,amscd,flafter,epsf}
\usepackage{graphics}
\usepackage{epsfig}
\usepackage{psfrag}

\newcommand\commentable[1]{#1}

\newcommand{\HF}{HF}

\newtheorem{theorem}{Theorem}[section]
\newtheorem{prop}[theorem]{Proposition}
\newtheorem{cor}[theorem]{Corollary}

\newtheorem{defn}[theorem]{Definition}

\newtheorem{remark}[theorem]{Remark}

\def\endproof{\relax\ifmmode\expandafter\endproofmath\else
  \unskip\nobreak\hfil\penalty50\hskip.75em\hbox{}\nobreak\hfil\bull
  {\parfillskip=0pt \finalhyphendemerits=0 \bigbreak}\fi}
\def\endproofmath$${\eqno\bull$$\bigbreak}
\def\bull{\vbox{\hrule\hbox{\vrule\kern3pt\vbox{\kern6pt}\kern3pt\vrule}\hrule}}

\newcommand{\R}{\mathbb{R}}

\newcommand{\C}{\mathbb{C}}

\newcommand{\Z}{\mathbb{Z}}

\newcommand{\OneHalf}{\frac{1}{2}}

\newcommand{\Zmod}[1]{\Z/{#1}\Z}

\newcommand{\cm}{\cdot}

\newcommand\relspinc{\underline{\spinc}}

\newcommand\Filt{{\mathcal F}_\spinc}

\newcommand\x{\mathbf x}

\newcommand\y{\mathbf y}

\newcommand\ModFlow{\mathcal M}

\newcommand\ModSphere{\ModFlow\left({\mathbb S}\longrightarrow 
\Sym^{g-1}(\Sigma_{1})\times \Sym^2(\Sigma_{2})\right)}
\newcommand\ModSpheres\ModSphere
\newcommand\CF{CF}

\newcommand\CFa{\widehat{CF}}
\newcommand\CFp{\CFb}
\newcommand\CFm{\CF^-}

\newcommand\HFp{\HFb}

\newcommand\HFm{\HF^-}
\newcommand\CFinf{CF^\infty}
\newcommand\HFinf{HF^\infty}
\newcommand\CFb{CF^+}
\newcommand\HFa{\widehat{HF}}
\newcommand\HFb{HF^+}

\newcommand\UnparModSp{\widehat \ModSp}
\newcommand\UnparModFlow\UnparModSp
\newcommand\Mod\ModSp

\newcommand{\spinc}{\mathfrak s}

\newcommand{\spinct}{\mathfrak t}

\newcommand\ModMaps{\mathcal M}
\newcommand\ModSp\ModMaps

\newcommand\Ta{{\mathbb T}_{\alpha}}

\newcommand\Tb{{\mathbb T}_{\beta}}

\newcommand\alphas{\mbox{\boldmath$\alpha$}}

\newcommand\betas{\mbox{\boldmath$\beta$}}

\newcommand\PerDom{\mathcal P}

\newcommand\spincrel\relspinc

\newcommand\CFK{CFK}
\newcommand\HFK{HFK}

\newcommand\CFKinf{\CFK^{\infty}}

\newcommand\HFKa{\widehat\HFK}

\newcommand\BasePt{w}
\newcommand\FiltPt{z}

\newcommand\Dual{\mathcal D}
\newcommand\Duality\Dual

\newcommand\EulerMeasure{\widehat\chi}

\newcommand\spinck{\mathfrak k}

\newcommand\ons{Ozsv{\'a}th and Szab{\'o}}
\newcommand\os{Ozsv{\'a}th-Szab{\'o}}

\newcommand\Mapn{{\widehat F}_{\!\sd n,m}}

\newcommand\tauxF{\tau_{[x]}(Y,[F],K)}

\newcommand\tauydual{\tau^*_{[y]}(Y,K)}
\newcommand\tauyup{\tau^*_{[y]}(Y,K)}

\newcommand\tauydualF{\tau^*_{[y]}(Y,[F],K)}

\newcommand\Cij[2]{C_\spinc\{i{#1},j{#2}\}}
\newcommand\Ci[1]{C_\spinc\{i{#1}\}}
\newcommand\SpinC{\mathrm{Spin}^c}
\newcommand\tauc{\tau_{\xi}(K)}
\newcommand\taucF{\tau_{\xi}([F],K)}
\newcommand\taux{\tau_{[x]}(Y,K)}
\newcommand\tauy{\tau_{[y]}(-Y,K)}

\newcommand\contact{c(\xi)}
\newcommand\filtYF[1]{ {\Filt}(Y,[F],K,{#1})}
\newcommand\filtY[1]{ {\Filt}(Y,K,{#1})}

\newcommand\Sym{\mathrm{Sym}}
\newcommand\sd{\mbox{-}}

\newcommand\sympW{W_K}
\newcommand\sympWbar{\overline{W}_K}

\commentable{

\title[{An invariant of knots in a contact manifold}] 
{An Ozsv{\'a}th-Szab{\'o} Floer homology invariant of knots in a contact manifold}
}

\author[Matthew Hedden]{Matthew Hedden}
\address{Department of
Mathematics, Massachusetts Institute of Technology\newline
\indent{\emailfont{mhedden@math.mit.edu}}}

\thanks{Matthew Hedden was supported by an NSF postdoctoral fellowship during the course of this work}

\begin{document}

\begin{abstract} Using the knot Floer homology filtration, we define invariants associated to a knot in a three-manifold possessing non-vanishing Floer co(homology) classes. In the case of the \os \ contact invariant we obtain an invariant of knots in a contact three-manifold. This invariant provides an upper bound for the Thurston-Bennequin plus rotation number of any Legendrian realization of the knot.   We use it to demonstrate the first systematic construction of prime knots in contact manifolds other than $S^3$ with negative maximal Thurston-Bennequin invariant.  Perhaps more interesting, our invariant provides a criterion for an open book to induce a tight contact structure. A corollary is that if a manifold possesses contact structures with distinct non-vanishing \os \ invariants, then any fibered knot can realize the classical Eliashberg-Bennequin bound in at most one of these contact structures.  
\end{abstract}

\maketitle

\section{Introduction}

A contact structure, $\xi$, on a closed oriented three-manifold, $Y$, is an oriented two-dimensional sub-bundle of the tangent bundle, $TY$, which is completely non-integrable.  This means there do not exist surfaces embedded in $Y$ whose tangent planes lie in $\xi$ in any open subset of the surface.  See \cite{Etnyre1} for an introduction.   It has been known for some time that there is a dichotomy between contact structures on a three-manifold: every contact structure falls into one of two classes, {\em overtwisted} or {\em tight}.  These classes are determined by the existence (in the case of overtwisted) or non-existence (in the case of tight) of an embedded disk whose interior is transverse to $\xi$ everywhere except one point, and whose boundary is tangent to $\xi$ .  A fundamental theorem of Eliashberg states that the overtwisted contact structures are classified by the homotopy type of the contact structure as a two-plane field.  Tight contact structures, on the other hand, have proved to be  much more difficult to understand and their classification is presently out of reach for a general three-manifold.  Since the definition of overtwisted involves the existence of a particular type of unknotted circle tangent to $\xi$ (unknotted in the sense that it bounds a disk), it may not be surprising to find that one of the ways in which tight contact structures differ from overtwisted involves knot theory.  To describe this distinction, we first recall some basic definitions from the theory of Legendrian knots. 

A knot which is everywhere tangent to $\xi$ is called {\em Legendrian}.  Given a Legendrian knot, $K$, we can form a push-off, $K'$, of $K$ using a vector field tangent to the contact planes but orthogonal to the tangent vector field of $K$.  If $K$ is null-homologous then the linking number lk$(K,K')$ is well-defined.  This linking number is called the Thurston-Bennequin number of $K$ and is denoted $tb(K)$.  It is immediate from the definition that $tb(K)$ is invariant under isotopy of $K$ through Legendrian knots, so-called {\em Legendrian isotopy}.

There is another easily defined integer-valued invariant of Legendrian knots.  Let $K$ as above be a null-homologous Legendrian knot with Seifert surface $F$.  Since the contact structure restricted to $F$, $\xi|_F$, is a real oriented two-dimensional vector bundle on a surface with boundary, it is necessarily trivial.  Picking a trivialization, $$\begin{CD} \tau: \xi|_F @>{\cong}>> F\times \R^2, \end{CD}$$ the tangent vector field to $K$ yields a map $u:S^1\rightarrow \R^2\sd\{0\}$.  We define the {\em rotation number of K}, $rot_F(K)$, to be the winding number of this map.  Note that the rotation number depends on our choice of Seifert surface, but only through its homology class $[F]\in H_2(Y\!-\!K;\Z)\cong H_2(Y;Z)$.  It is straightforward to verify that the rotation number, like the Thurston-Bennequin number, is invariant under Legendrian isotopy.  We refer to the Thurston-Bennequin and rotation numbers of a Legendrian knot as its ``classical'' invariants.

A fundamental theorem of Eliashberg \cite{Eliashberg} states that for tight contact structures, the classical invariants of Legendrian knots are constrained by the topology of the three-manifold:

\begin{theorem} \label{thm:Eliashberg} (Eliashberg-Bennequin inequality)
	Let $\xi$ be a {\em tight} contact structure on a three-manifold, $Y$.  Then for a null-homologous knot $K\hookrightarrow Y$ and Seifert surface, $F$, we have $$tb(\tilde{K})+|rot_F(\tilde{K})|\le 2g(F)-1,$$
	\noindent where $\tilde{K}$ is any Legendrian representative of $K$.
\end{theorem}

This is in stark contrast with overtwisted contact structures, where a given knot type has Legendrian representatives with arbitrarily large classical invariants. (Bennequin \cite{Bennequin} proved the above inequality for the standard tight contact structure on the three-sphere, $(S^3,\xi_{std})$, explaining the name.)  

Since the Eliashberg-Bennequin inequality, much work has been done to further constrain the classical invariants of Legendrian knots.  However, the work has primarily addressed the special case of Legendrian knots in $(S^3,\xi_{std})$.  The primary reason for the focus on knots in $(S,\xi_{std})$ is due to the fact that, for such knots, the classical invariants have a combinatorial description in terms of a particular type of projection of $K$ to $\R^2$, the {\em front projection}.  The combinatorics of such diagrams share some properties with various combinatorially defined knot invariants, e.g. the HOMFLY and Kauffman polynomials and Khovanov homology, and the best bounds for the classical invariants of knots in $(S^3,\xi_{std})$ come from these combinatorial knot invariants.  

For contact manifolds other than $(S^3,\xi_{std})$, much less is known about the classical invariants of Legendrian knots.  For Stein fillable contact structures, Eliashberg's bound was improved by Lisca and Matic \cite{LiscaMatic} (see also Akbulut and Matveyev \cite{Akbulut1}) and recently Mrowka and Rollin \cite{Mrowka} extended this to tight contact structures with non-vanishing Seiberg-Witten contact invariant.  An analogous theorem was proved for the \os \ contact invariant by Wu \cite{Wu}.  In both cases, the theorems replaced the genus of the Seifert surface by the genus of a surface properly embedded in a four-manifold bounded by the three-manifold.
It is important to note that aside from $(S^3,\xi_{std})$, all known bounds for the Thurston-Bennequin and rotation numbers of Legendrian knots involve $2g(F)-1=-\chi (F)$, for a surface-with-boundary, $F$, and hence are necessarily greater than or equal to $\sd 1$.  

The primary purpose of this paper is to introduce an integer-valued invariant $\tauc$ of a quadruple $(Y,\xi,[F],K)$ which will replace $g(F)$ in the Eliashberg-Bennequin inequality. Here $(Y,K)$ is a null-homologous knot, $[F]$ a homology class of Seifert surface, and $\xi$ a contact structure.  The precise definition of $\tauc$ will be given in the next section, but roughly speaking it uses the knot Floer homology filtration associated to $(Y,[F],K)$ together with the \os \ contact invariant, $\contact$.  In the case that $\contact\ne 0$ we will prove the following:

	\begin{theorem} \label{thm:tbbounds} 
	Let $(Y,\xi)$ be a contact three-manifold with non-trivial \os \ contact invariant.  Then for a null-homologous knot $K\hookrightarrow Y$ and Seifert surface, $F$ we have $$tb(\tilde{K})+|rot_F(\tilde{K})|\le 2\tauc-1,$$
	\noindent where $\tilde{K}$ is any Legendrian representative of $K$.
\end{theorem}
\begin{remark} Note that, in general, $rot_F(K)$ and $\tauc$ both depend on $[F]\in H_2(Y;\Z)$.  However, it will be shown that if $F$ and $F'$ are two Seifert surfaces then $$2\tau^F_{\xi}(K)-2\tau^{F'}_{\xi}(K)= \langle c_1(\xi), [F\!-\!F'] \rangle, $$ where $c_1(\xi)$ is the first Chern class of the contact structure or, equivalently, its Euler class.  On the other hand, it is easy to see that the rotation number depends on $[F]$ in the same way.
\end{remark}
	
By a theorem of \ons, the non-vanishing of the contact invariant implies tightness of $\xi$ and it will be immediate from the definition and an adjunction inequality that $\tauc\le g(F)$.
Thus the bounds obtained above will be at least as good as the Eliashberg-Bennequin bound. Indeed, in an upcoming paper we will show that the above bound is also as good as that provided by Wu \cite{Wu} (or Mrowka and Rollin).  Unlike $g(F)$, however, $\tauc$ can be negative and hence provides the first general method for determining prime knot types in contact manifolds other than $S^3$ whose classical Legendrian invariants are constrained to be negative.  Here, prime means that the only decomposition of $(Y,K)$  as a connect sum $(Y_1,K_1)\# (Y_2,K_2)$ is when one of the summands is $(S^3,\mathrm{unknot})$.  The primeness condition here is essential, since the combinatorial techniques described above can be adapted to the situation when we form the connected sum of a knot in $S^3$ and an unknot (i.e. a knot bounding a disk) in an arbitrary tight contact manifold.  Indeed, we will show that there exist prime knots in any contact manifold with $\contact\ne 0$ which have classical invariants constrained to be arbitrarily negative.  More precisely, we have

\begin{theorem} \label{thm:arbneg}
	Let $(Y,\xi)$ be a contact manifold with non-trivial \os \ contact invariant.  Then for any $N>0$, there exists a prime knot $K\hookrightarrow Y$ such that $$tb(\tilde{K})+|rot_F(\tilde{K})| \le -N,$$
	\noindent for any Legendrian representative, $\tilde{K}$, of $K$, and any Seifert surface, $F$.
\end{theorem}

The proof draws on results of \cite{CableII} which determine the behavior of the knot Floer homology filtration under a certain satellite operation called cabling.  In particular, negative upper bounds on $\tauc$ of sufficiently negative cables of any knot  can easily be achieved.  The precise statement of these bounds is described in Section \ref{sec:proofs}.

Another application of $\tauc$ involves contact structures induced by open book decompositions of a given three-manifold. 
 Recall that a fibered knot is a triple of data $(Y,F,K)$ consisting of a knot $(Y,K)$ and a surface $F$ with $\partial F=K$ for which we have the following identification:
 $$ Y\sd\nu(K) \cong \frac{F\times [0,1]}{\{(x,0)\simeq (\phi(x),1)\}}, $$
 \noindent where  $\phi$ is a diffeomorphism of $F$ fixing $\partial F$ and $\nu(K)$ is a neighborhood of $K$.  The decomposition of $Y\sd\nu(K)$ given above produces a decomposition of $Y$:
 $$ Y \cong \frac{F\times [0,1]}{\{(x,0)\simeq (\phi(x),1)\}} \cup D^2\times S^1,$$
 \noindent where we identify $\partial F\times \{p\}$ with $\{q \in \partial D^2\}\times S^1$ and $\{p'\in\partial F\}\times S^1$ with \newline $ \partial D^2\times\{q'\in S^1\}$.   Such a decomposition is called an {\em open book decomposition} of $Y$.  There is a well-known construction due to Thurston and Winkelnkemper \cite{Thurston} which associates a canonical contact structure on $Y$ to an open book decomposition.  In this way, we can associate a contact structure  to a fibered knot $(F,K)$.  Let $$\xi_{(F,K)}:=\mathrm{contact \ structure \ associated \ to \ the \ fibered \ knot\ } (F,K). $$  Given a contact manifold $(Y,\xi)$ and fibered knot $(F,K)$ one can ask whether there is a relationship between the classical invariants of $K$ in $\xi$ and the contact structure $\xi_{(F,K)}$.  The following theorem indicates that such a relationship exists, and provides a sufficient condition for $\xi_{(F,K)}$ to be tight in terms of the classical invariants of Legendrian representatives of $K$ in $\xi$.  
 
 \begin{theorem}\label{thm:fibered}	 Let $(Y,\xi)$ be a contact structure with non-trivial \os \ contact invariant.  Let $(F,K)$ be a fibered knot which realizes the Eliashberg-Bennequin bound in $\xi$.  That is, there exists a Legendrian representative, $\tilde{K}$, of $K$ such that:
	\begin{equation}\label{eq:EB}
	 tb(\tilde{K})+|rot_F(\tilde{K})|= 2g(F)-1.
 \end{equation}
 Then the contact structure associated to $(F,K)$ by the Thurston-Winkelnkemper construction, $\xi_{(F,K)}$,  is tight.  Furthermore, the \os \ contact invariants of $\xi_{(F,K)}$ and $\xi$ are identical. That is, $c(\xi_{(F,K)})=\contact$. 
 \end{theorem}
 
 Note that as a special case of the above theorem we have that a fibered knot in $S^3$ with $TB(K)=2g(K)-1$ induces the standard tight contact structure (here $TB(K)$ is the maximal Thurston-Bennequin number over all Legendrian representatives of $K$.)  
We also have the immediate corollary
\begin{cor}
	Let $Y$ be a three manifold and $\xi_1,\ldots, \xi_i$ be contact structures with distinct non-trivial \os \ invariants. That is, $c(\xi_i)\ne c(\xi_j)$ unless $i=j$.  Then, given a fibered knot $(F,K)$, the equality
 $$tb(\tilde{K})+|rot_F(\tilde{K})|= 2g(F)-1$$
 \noindent can hold in at most one of $\xi_j$.
 \end{cor}

 It is interesting to note that a fibered knot $(F,K)$ always has a Legendrian representative in $\xi_{(F,K)}$  realizing the Eliashberg-Bennequin bound \eqref{eq:EB}. Indeed, the construction of $\xi_{(F,K)}$ presents $K$ as a transverse knot satisfying $sl_F(K)=2g(F)-1$, and any Legendrian push-off of $K$ will satisfy \eqref{eq:EB}. (Here, $sl_F(K)$ is the self-linking number of the transverse knot, taken with respect to the fiber surface, $F$.  The stated equality follows from the observation that the characteristic foliation of $F$ in $\xi_{(F,K)}$ has no negative singularities.)
 Thus Theorem \ref{thm:fibered} and its corollary seem to indicate a surpising ``preference'' of fibered knot for its own contact structure, $\xi_{(F,K)}$, at least when we restrict attention to the subset of tight contact structures distinguished by their \os \ invariants.   It also leads to the following  
 
 \bigskip
 
\noindent  {\bf Question:} {\em It is known that there are tight contact structures with trivial \os \  invariant \cite{Ghiggini}.  However, one can ask if the conclusion of  Theorem $1.4$ holds if $(Y,\xi)$ is only assumed to be tight. That is,  does the existence of a Legendrian representative of a fibered knot $(F,K)$ in a tight contact structure $(Y,\xi)$  satisfying Equation \eqref{eq:EB} imply that $\xi_{(F,K)}$ is tight? If so, what is the relationship between $\xi_{(F,K)}$ and $\xi$?}

 \begin{remark}
  In another direction, we expect $\tauc$ to provide an obstruction to a knot $(Y,K)$ arising as the boundary of a properly embedded $J$-holomorphic curve in a symplectic filling of $(Y,K)$.  We will return to this point in an upcoming paper.  
  \end{remark}
\bigskip
  
\noindent {\bf Outline:} The organization of the paper is as follows. In the next section we spend a considerable amount of time setting up notation, reviewing basic properties of \os \ Floer homology for three-manifolds, its refinement for null-homologous knots, and the construction and properties of the contact invariant, $\contact$.  The main purpose of this section is to define several invariants associated to a knot in a three-manifold possessing non-vanishing Floer (co)homology classes.  The invariant $\tauc$ will be the special case of one of these invariants, when the Floer cohomology class is the contact invariant.  Section \ref{sec:Properties} establishes key properties of the invariants which generalize analogous properties of the \os \ concordance invariant.   Together with the results of \cite{CableII}, these properties will be used in Section \ref{sec:proofs} to prove the theorems. \bigskip

\noindent{\bf Acknowledgments:}   The original motivation for this work came from Olga Plamenev-skaya's paper \cite{Olga}, which established Theorem \ref{thm:tbbounds} for the  \os \ concordance invariant.  I have enjoyed and benefited from conversations with many people regarding the ideas presented here, among them John Etnyre, Paolo Ghiggini, Tom Mrowka, Peter Ozsv{\'a}th, Andr{\'a}s Stipsicz, and Hao Wu.  Special thanks go to Tim Perutz for pointing out an algebraic oversight in an earlier version of this work, and to Tom Mark for many useful comments and suggestions, and especially for his interest and help in dealing with the aformentioned oversight.

\section{Background on \os \ theory}
In this section we introduce and recall background on various aspects of the Floer homology package developed by \ons \ over the past several years.  All chain complexes will be over the field $\Z/2\Z$.  Due to the breadth of the theory, this section may not be sufficient for a complete understanding of the \os \ machinery, but we include it here to establish notation and recall the main results and structures of the theory which will be used. Much of the section can be skipped by the reader familiar with \os \ theory.  However, for such a reader, we call attention to Definitions \ref{defn:tau_x}, \ref{defn:tau_xdual},  and \ref{defn:tauc}.  These are the definitions of the invariants $\tau_{[x]}(Y,K)$, $\tau^*_{[y]}(Y,K)$, and  $\tauc$, respectively.  The idea behind each invariant is same as that of the \os \ concordance invariant or the Rasmussen $s$ invariant - a knot induces a filtration of a certain (co)chain complex and each invariant measures when the (co)homology of the subcomplexes in the filtration start to hit specific (co)homology classes.  The reason for multiple invariants is that in \os \ theory a knot induces a filtration on both the chain and cochain complexes associated to $Y$.  Moreover, the contact invariant $c(Y,\xi)$ is really an element of the Floer cohomology of $Y$, and hence we need an invariant, $\tau^*_{[y]}(Y,K)$, associated to a knot $K$ and a Floer cohomology class, $[y]$.  

Aside from these definitions, the only original material presented here is Property $4$ of the contact invariant, which is the behavior of the contact invariant under connected sums.  Though this property is expected and its proof straightforward, its appearance here is the first that we know of and may be of independent interest.  The rest of this section draws heavily on several articles of \ons \ \cite{FourBall,Contact,Knots}, and in some places we have simply adapted their work with notational changes - we stress that our purpose is to collect relevant results and establish notation. 

\subsection{The Knot Floer homology filtration}
\label{subsec:Filtration}
To a closed oriented three-manifold $Y$, equipped with a Spin$^c$ structure, $\spinc$, \ons \ defined several chain complexes, $\CFinf(Y,\spinc),\CFp(Y,\spinc),\CFm(Y,\spinc),\CFa(Y,\spinc)$ \cite{HolDisk}.  The homologies of these chain complexes, denoted $\HFinf(Y,\spinc),\HFp(Y,\spinc),\HFm(Y,\spinc),\HFa(Y,\spinc)$ were proved to be invariants of the pair $(Y,\spinc)$.   Associated to a null-homologous knot $K\hookrightarrow Y$, a choice of Seifert surface, $F$, and a Spin$^c$ structure, $\spinc$, they subsequently defined filtered versions of the above chain complexes, and proved that the filtered chain homotopy types of these chain complexes are invariants of the quadruple $(Y,[F],K,\spinc)$ (here $[F]\in H_2(Y\!-\!K;\Z)\cong H_2(Y;\Z)$ is the homology class of the Seifert surface).   We discuss the most general of these complexes, denoted $\CFKinf(Y,[F],K,\spinc)$.  Each of the other \os \ Floer chain complexes for knots and three-manifolds can be derived from this chain complex, and so we describe it first. We then discuss how to obtain some of the other invariants from it.   This approach is historically backwards, but our main purpose here it set up notation and collect  properties of the  chain complexes we use throughout the text. For a complete discussion we refer the interested reader to \cite{HolDisk,HolDiskTwo,Knots,Ras1}  and to \cite{Survey} for a survey. 

Fix a doubly-pointed (admissible) Heegaard diagram $(\Sigma_g,\alphas,\betas,w,z)$
for the knot $(Y,K)$ (see Definition $2.4$ of \cite{Knots} and Definition $4.10$ of \cite{HolDisk}) and consider the $g$-fold symmetric product  $\Sym^g(\Sigma_g)$, with two
 tori \begin{eqnarray*}
\Ta=\alpha_1\times\cdots\times\alpha_g&{\text{and}}&
\Tb=\beta_1\times\cdots\times\beta_g.
\end{eqnarray*} 
By an isotopy of the attaching curves, these tori intersect transversely in a finite number of points.  
\noindent In Section $2.3$ of \cite{Knots} \ons \ define a map $$\underline{\spinc}: \{\Ta\cap \Tb\} \rightarrow\mathrm{Spin}^c(Y_0(K))\simeq \mathrm{Spin}^c(Y)\times \Z,$$
\noindent which assigns to each intersection point  $\x\in \{\Ta\cap \Tb\}\subset \mathrm{Sym}^g(\Sigma)$ a Spin$^c$ structure 
on the zero-surgery of $Y$ along $K$, $Y_0(K)$. The projection from $\SpinC(Y_0(K))$ to $\SpinC(Y)$ is obtained by first restricting $\spinc$ to $Y\!-\!K$,
and then uniquely extending it to $Y$. Projection to the second factor comes from evaluation $\OneHalf
 \langle c_1(\spinc), [\widehat F]\rangle$, where $\widehat F$ denotes a surface in $Y_0(K)$ obtained by capping off a fixed Seifert surface, $F$, for $K$ with the meridian disk of the solid torus glued to $Y\!-\!K$ in the surgery. We say that a Spin$^c$ structure on $Y_0(K)$ {\em extends} $\spinc\in \mathrm{Spin}^c(Y)$ if projection onto the factor of Spin$^c(Y_0(K))$ corresponding to Spin$^c(Y)$ is equal to $\spinc$.

\begin{remark} More generally,  to an intersection point $\x$, the map $\underline{\spinc}$ assigns a relative $\SpinC$ structure $\underline{\spinc}(\x)$ on the knot complement. However, for null-homologous knots relative $\SpinC$ structures can be identified with $\SpinC$ structures on $Y_0(K)$.\end{remark}

 Fix  $\spinc\in \SpinC(Y)$ and a homology class of Seifert surface $[F]\in H_2(Y;\Z)$.  Now let $\spinc_0$ denote the unique $\SpinC$ structure in $\SpinC(Y_0(K))$ such that $\spinc_0$ extends $\spinc$ and satisfies $\OneHalf \langle c_1(\spinc_0),[\widehat F]\rangle=0$.  The chain complex $\CFKinf(Y,[F],K,\spinc)$ is then generated (as a $\Z/2\Z$ vector space) by  triples $[\x,i,j]$  satisfying the constraint 
\begin{equation}\label{eq:constraint}
\underline{\spinc}(\x) + (i-j)\mathrm{PD}(\mu)=\spinc_0.
\end{equation}
\noindent Here $\mathrm{PD}(\mu)\in H^2(Y_0(K);\Z)$ is the Poincar{\'e} dual to the meridian of $K$ and addition is meant to signify the action of $H^2(Y_0(K);\Z)$ on $\SpinC(Y_0(K))$. The constraint depends on the choice of Seifert surface but only through its homology class $[F]\in H_2(Y;\Z)$.  Indeed, this is the only place where the Seifert surface appears in the knot Floer homology construction.  Furthermore, Propositions \ref{prop:filtdiff} and \ref{prop:filtdiffhat} below show that the effect of varying $[F]$ can be easily understood in terms of the algebraic topology of $Y$.   Thus, when $[F]$ is clear from the context e.g. $Y$ is a rational homology sphere and $[F]=0$, or when it becomes notationally cumbersome, we will omit it from the discussion.

The boundary operator on $\CFKinf(Y,[F],K,\spinc)$ is defined by
 $$\partial
[\x,i,j] =
\sum_{\y\in\Ta\cap\Tb} \sum_{\{\phi\in\pi_2(\x,\y)\}}
\#\left(\frac{\ModFlow(\phi)}{\R}\right) [\y,i-n_\BasePt(\phi),j-n_\FiltPt(\phi)],$$ 
\noindent where $\#\left(\frac{\ModFlow(\phi)}{\R}\right)$ denotes a count, modulo $2$,  of points in the moduli space of unparameterized pseudo-holomorphic Whitney disks, $\phi$, with boundary conditions specified by $\x,\y$ and $\Ta$, $\Tb$.  The integers $n_\BasePt(\phi),n_\FiltPt(\phi)$ are intersection numbers between the image of $\phi$ in Sym$^g(\Sigma)$ with the codimension one subvarieties $\{\BasePt\}\times\mathrm{Sym}^{g-1}(\Sigma)$, $\{\FiltPt\}\times\mathrm{Sym}^{g-1}(\Sigma)$.  See Sections $2$ and $4$ of \cite{HolDisk} for relevant details and definitions regarding the boundary operator, and Section $3$ of \cite{HolDisk} for its analytical underpinnings.

If we define a partial ordering on $\Z\oplus \Z$ by the rule that $(i,j)\le (i',j')$ if $i\le i'$ and $j\le j'$, then  a {\em $\Z\oplus \Z$-filtered chain complex} is by definition a chain complex $C_*$ equipped with a map:
$$ \mathcal{F}: C_* \rightarrow \Z\oplus \Z,$$ 
\noindent such that the differential $\partial$ respects $\mathcal{F}$ in the sense that  $$\mathcal{F}(\partial(\x)) \le \mathcal{F}(\x) \ \mathrm{for\ every}\ \x\in C_*.$$

\noindent From its construction, it is immediate that $\CFKinf(Y,[F],K,\spinc)$ is a  $\Z\oplus \Z$-filtered chain complex - for a generator we simply define $\Filt([\x,i,j])=(i,j)$.  More generally, for a chain $c=\underset{k}{\Sigma} [\x_k,i_k,j_k]$,  the filtration is given by $\Filt(c)=(\underset{k}{\mathrm{max}}\ i_k,\underset{k}{\mathrm{max}}\ j_k)$.  

Now the Whitney disks counted in  $\#\left(\frac{\ModFlow(\phi)}{\R}\right)$ have pseudo-holomorphic representatives, and hence the quantities $n_\BasePt(\phi)$ and $n_\FiltPt(\phi)$ are necessarily positive - indeed the submanifolds  $\{\BasePt\}\times\mathrm{Sym}^{g-1}(\Sigma)$, $\{\FiltPt\}\times\mathrm{Sym}^{g-1}(\Sigma)$ are pseudo-holomorphic and thus intersect the image of pseudo-holomorphic Whitney disks positively (see Lemma $3.2$ of \cite{HolDisk}). Hence $\Filt$ equips $\CFKinf(Y,[F],K,\spinc)$ with a $\Z\oplus\Z$-filtration.  Theorem $3.1$ of \cite{Knots} proved that the $\Z\oplus\Z$-filtered chain homotopy type of $\CFKinf(Y,[F],K,\spinc)$ is an invariant of the quadruple $(Y,[F],K,\spinc)$.  Indeed this is the primary knot invariant in \os \ theory and is quite powerful - it has been shown that the filtered chain homotopy type of $\CFKinf(Y,[F],K,\spinc)$ determines the genus of $K$ \cite{GenusBounds}, whether $K$ is fibered \cite{Ghiggini,Ni1, Juhasz}, can be used to determine the Floer homology of three-manifolds obtained by surgery along $(Y,K)$ \cite{IntegerSurgeries,RationalSurgeries}, and has applications to determining the smooth four-genera of knots in $S^3$ \cite{FourBall}.   Note that $\Filt$ depends on $[F]$ through Equation \eqref{eq:constraint}, but only up to an overall shift which we now make precise.

\begin{prop}\label{prop:filtdiff}
	Let $F,F'$ be two Seifert surfaces for a knot $K\hookrightarrow Y$.   Fix $\x\in \Ta\cap\Tb$ and let $[\x,i,j]\in\CFKinf(Y,[F],K,\spinc)$ and $[\x,i',j']\in\CFKinf(Y,[F'],K,\spinc)$ be generators.  Then we have the relation:
	\begin{equation}\label{eq:filtdiff} (i-i')+(j'-j)= -\OneHalf \langle c_1(\spinc),[F\!-\!F'] \rangle, \end{equation}
\noindent	where 	$[F\!-\!F']\in H_2(Y;\Z)$ is the difference of the homology classes of $F$ and $F'$, and $c_1(\spinc)$ is the first Chern class of the $\SpinC$ structure, $\spinc$.
\end{prop}
\begin{proof}
	Equation \eqref{eq:constraint} determines the triples which generate $\CFKinf(Y,[F],K,\spinc)$ and $\CFKinf(Y,[F'],K,\spinc)$.  The place in this equation where $[F]$ places a role is in the choice of $\spinc_0\in \SpinC(Y_0(K))$ extending $\spinc\in \SpinC(Y)$.   Now $F$ and  $F'$ yield $\SpinC$ structures $\spinc_0$ and $\spinc_0'$, respectively, which extend $\spinc$ and satisfy \begin{eqnarray*}
 \langle c_1(\spinc_0),[\widehat F]\rangle=0 &{\text{resp.}}&
 \langle c_1(\spinc_0'),[\widehat F']\rangle=0.
\end{eqnarray*} 
\noindent In order for both these equalities to hold,  $\spinc_0$ and $\spinc_0'$ are forced to be related by  $$\spinc_0'-\spinc_0=\OneHalf \langle c_1(\spinc),[F\!-\!F'] \rangle\cm \mathrm{PD}(\mu)\in H^2(Y_0(K);\Z).$$ Now \eqref{eq:constraint} requires that
$$\underline{\spinc}(\x) + (i-j)\mathrm{PD}(\mu)=\spinc_0$$
$$\underline{\spinc}(\x) + (i'-j')\mathrm{PD}(\mu)=\spinc_0',$$
\noindent for the respective choice of Seifert surfaces.  Subtracting the second second equation from the first yields the desired relation.

\end{proof} 

Much of the power of the filtered chain homotopy type of $\CFKinf(Y,[F],K,\spinc)$  lies in our ability to construct new topological invariants by restricting attention to subsets  $C_\spinc\subset\CFKinf(Y,[F],K,\spinc)$ whose $\Filt$-values  satisfy various numerical constraints.  If the differential on  $\CFKinf$ restricts to a differential on the chosen subset (i.e. $(\partial|_{C_\spinc})^2=0$) then the homology of $C_\spinc$ with respect to the restricted differential will be an invariant of $(Y,[F],K,\spinc)$.  For instance, we can examine the set $$\Ci{=0} \subset \CFKinf(Y,[F],K,\spinc),$$
\noindent consisting of generators of the form $[\x,0,j]$ for some $j\in\Z$. This set naturally inherits a differential from $\CFKinf(Y,[F],K,\spinc)$, since it is a subcomplex of the quotient complex $\frac{\CFKinf}{ \Ci{<0}}$.  We have the isomorphism of chain complexes $$\Ci{=0} \cong \CFa(Y,\spinc),$$
\noindent (which the uninitiated reader can take as the definition of $\CFa(Y,\spinc)$).  Thus we recover the ``hat'' Floer homology of $(Y,\spinc)$ from $\CFKinf(Y,[F],K,\spinc)$.  Furthermore, by restricting $\Filt$ to $\Ci{=0}$ we equip $\CFa(Y,\spinc)$ with a $\Z$-filtration.  In particular, if we denote by $\filtYF{m}$ the subcomplex of $\CFa(Y,\spinc)$:
$$\filtYF{m} := \Cij{=0}{\le m},$$

\noindent then we have the finite sequence of inclusions:
$$0=\filtYF{\sd j}\hookrightarrow \filtYF{\sd j+1}\hookrightarrow \ldots \hookrightarrow \filtYF{n}=\Ci{=0}.$$
\noindent (Finiteness of the above sequence follows from the fact the number of intersection points $\x\in \{\Ta\cap\Tb\}$ is finite.)  Proposition \ref{prop:filtdiff} indicates that the dependence of this filtration on  $[F]$ is given by: 

\begin{prop}\label{prop:filtdiffhat}	Let $F,F'$ be two Seifert surfaces for a knot $K\hookrightarrow Y$.  Let $\Filt$ and $\Filt'$ denote the resulting filtrations of $\CFa(Y,\spinc)$ induced by $([F],K)$ and $([F'],K)$, respectively.  Then, for fixed $\x\in \CFa(Y,\spinc)$, we have:
$$ \Filt'(\x) = \Filt(\x) -\OneHalf \langle c_1(\spinc),[F\!-\!F'] \rangle.$$
	\end{prop}

	\begin{proof} The relation follows immediately from Equation \eqref{eq:filtdiff}.  Specifically, a generator  $\x\in\CFa(Y,\spinc)$ corresponds either to a triple $[x,0,j]\in \CFKinf(Y,[F],K,\spinc)$ or to a triple $[\x,0,j'] \in \CFKinf(Y,[F'],K,\spinc)$. In terms of these triples, $\Filt(\x)=j$ and $\Filt'(\x)=j'$.  Equation \eqref{eq:filtdiff} now shows that $j'-j=    -\OneHalf \langle c_1(\spinc),[F\!-\!F'] \rangle,$ proving the proposition.\end{proof}

Some particularly interesting invariants derived from the filtration of $\CFa(Y,\spinc)$ are the homology groups of the successive quotients (the associated graded homology groups), $H_*(\frac{\filtYF{m}}{\filtYF{m\sd1}})$ which we denote by $\HFKa_\spinc(Y,[F],K,m)$.  These are the so-called ``knot Floer homology groups'' of $(Y,[F],K,\spinc)$.  For the case of knots in the three-sphere, the weighted Euler characteristic of these groups is the classical Alexander-Conway polynomial, $\Delta_K(T)$ of the knot (see \cite{Knots,Ras1}):
$$\sum_{m} \chi \left(\HFKa(S^3,K,m)\right) \cm T^m =\Delta_K(T).$$

 \noindent In terms of the above subcomplexes, we can define a numerical invariant of a knot in a three-manifold with non-zero Floer homology class $[x]\ne 0\in\HFa(Y,\spinc)$ as follows.  Let $I_m$ denote the map on homology induced by the inclusion:
$$ \iota_m: \filtYF{m}\hookrightarrow \CFa(Y,\spinc).$$

\noindent Then given $[x]\ne 0\in \HFa(Y,\spinc)$ we have the following integer associated to $(Y,[F],K,\spinc)$:
\begin{defn}
	\label{defn:tau_x}
$$	\tau_{[x]}(Y,[F],K)=\mathrm{min}\{m\in\Z|\ [x]\subset \mathrm{Im} \ I_m\}. $$\end{defn}
\begin{remark} \label{remark:dependence} Note that our notation suppresses $\spinc\in \SpinC(Y)$. In light of Proposition \ref{prop:filtdiffhat}, the dependence on $[F]$ is given by:
	$$\tau_{[x]}(Y,[F],K)-\tau_{[x]}(Y,[F'],K)= 	\OneHalf \langle c_1(\spinc),[F\!-\!F'] \rangle.$$
\end{remark}
\noindent  It follows immediately from the fact that the $\Z\oplus\Z$ filtered chain homotopy type of $\CFKinf$ is an invariant of $(Y,[F],K,\spinc)$, that $\tau_{[x]}(Y,[F],K)$ is also an invariant of  $(Y,[F],K,\spinc)$.  This paper will focus on the case when $[x]$ is the \os \ contact invariant $c(\xi)\in \HFa(-Y)$, described in Subsection \ref{subsec:contact} below.  Since $c(\xi)$ is an element of the Floer homology of the three-manifold with reversed orientation, $\HFa_*(-Y)$, and this group can be identified with the Floer cohomology, $\HFa^*(Y)$, it will be useful to be able to ``dualize''  $\tau_{[x]}(Y,[F],K)$ in an appropriate sense.  For these purposes, we digress to discuss the precise behavior of (knot) Floer homology under orientation reversal of the underlying three-manifold.

\subsection{Orientation Reversal of Y}

We begin by  recalling the following proposition:
\begin{prop}\label{prop:Duality}(Proposition $2.5$ of \cite{HolDiskTwo}) Let $Y$ be an oriented three-manifold equipped with a Spin$^c$ structure, $\spinc$, and let $-Y$ denote the manifold with reversed orientation, then we have a natural chain homotopy equivalence:
	$$ \CFa^*(Y,\spinc):=(\mathrm{Hom}(\CFa(Y,\spinc),\Z/2\Z),\delta)\cong \CFa_*(-Y,\spinc) $$
\end{prop}
\begin{remark} The term on the left is the dual complex associated to the chain complex $\CFa(Y,\spinc)$, hence the Floer homology of $-Y$ is isomorphic to the Floer cohomology of $Y$. Throughout, we will denote dual complexes with an upper star and, like our chain complexes, these will always be with $\Z/2\Z$ coefficients so to avoid dealing with $Ext$ terms (note that at the time of writing, the author knows of no examples of knots or three-manifolds with torsion in the ``hat'' versions of \os \ Floer homology, though torsion has been found in the $\pm,\infty$ varieties \cite{Mark}). 
\end{remark}
\begin{proof}  If we fix a Heegaard diagram $(\Sigma,\alphas,\betas,w)$ for $Y$,  a Heegaard diagram for $-Y$ is obtained by either reversing the orientation of the Heegaard surface,  $(-\Sigma,\alphas,\betas,w)$, or switching the roles of the $\alpha$ and $\beta$ curves,  $(\Sigma,\betas,\alphas,w)$.  In either case there is an identification of intersection points $ \x\in \Ta\cap\Tb\subset \mathrm{Sym}^g(-\Sigma)$  (respectively  $ \x\in \Tb\cap\Ta\subset \mathrm{Sym}^g(\Sigma)$) with those in $\Ta\cap \Tb\subset \mathrm{Sym}^g(\Sigma)$.   Moreover, upon switching $\alphas$ and $\betas$, $J_s$-holomorphic Whitney disks in Sym$^g(\Sigma)$ connecting $\x$ to $\y$ for the chain complex coming from  $(\Sigma,\alphas,\betas,w)$ are identified with $J_s$-holomorphic Whitney disks in Sym$^g(\Sigma)$ connecting $\y$ to $\x$ in the chain complex for $(\Sigma,\betas,\alphas,w)$.   This yields the identification of the proposition.  For the case where the orientation of $\Sigma$ is reversed, we can alternatively prove the proposition as follows:  Fix $\phi\in\pi_2(\x,\y)$, and let  $\overline{\phi}\in\pi_2(\y,\x)$ denote the homotopy class of the disk in Sym$^g(-\Sigma)$ obtained from $\phi$ by pre-composing with complex conjugation in $\C$.  Then there is an identification of moduli spaces	$$\ModFlow_{\overline{J}_s}(\overline{\phi})\cong \ModFlow_{J_s}(\phi),$$
	\noindent where $\overline{J}_s$ denotes the almost complex structure on Sym$^g(-\Sigma)$ obtained from $J_s$ by conjugation.  This identification of moduli spaces provides an alternative proof of the proposition.  Note that since conjugation takes place in both $\C$ and Sym$^g(-\Sigma)$, intersection numbers are unaffected, i.e.  $n_z(\phi)=n_z(\overline{\phi})$.
\end{proof}

We will also have need for the behavior of the knot filtration $\Filt(Y,[F],K)$ under orientation reversal of $Y$.  For the present paper, the following proposition will be sufficient:
\begin{prop}\label{prop:knotduality}(compare Proposition $3.7$ of \cite{Knots}) 
Consider the short exact sequence of chain complexes for $\Filt(Y,K,m)$:
$$
\begin{CD}
	0 @>>> \Filt(Y,K,m) @>{\iota_m}>> \CFa(Y,\spinc) @>{p_m}>> Q_\spinc(Y,K,m) @>>> 0.
\end{CD}
$$
There is a natural identification:
\begin{equation}
	\label{eq:knotduality}
\begin{CD}
	0 \leftarrow \Filt^*(Y,K,m) @<{\iota_m^*}<< \CFa^*(Y,\spinc) @<{p_m^*}<< Q_\spinc^*(Y,K,m) \leftarrow 0 \\
  @V{\cong}VV @V{\cong}VV @V{\cong}VV  \\ 
 	0\leftarrow  Q_\spinc(\sd Y,K,\sd m\sd 1) @<{\sd p_{\sd m\sd 1}}<< \CFa(\sd Y,\spinc) @<{\sd \iota_{\sd m\sd 1}}<< \Filt(\sd Y,K,\sd m\sd 1) \leftarrow 0 \\
\end{CD}
\end{equation}
\noindent where the top row is the dual of the first short exact sequence and the bottom is the short exact sequence corresponding to $\Filt(-Y,K,-m-1)$
\end{prop}
\begin{remark} Here, and throughout, we denote by $-\iota_m$ and $-p_m$ the inclusion and projection maps for the short exact sequence corresponding to $\Filt(-Y,K,m)$, and $-I_m$ and $-P_m$ for the corresponding maps on homology.  Note that we have suppressed $[F]$ to simplify notation.
\end{remark}
\begin{proof} Upon dualizing, it is immediate that subcomplexes become quotient complexes, and conversely.  Thus it remains to see that we can identify filtrations as stated.  As in Proposition \ref{prop:Duality}, we can obtain a doubly-pointed Heegaard diagram for $(-Y,K)$ from one representing $(Y,K)$ by either reversing the orientation of $\Sigma$ or switching the roles of the $\alpha$ and $\beta$ curves.  In either event, the net result was that Whitney disks reversed direction (i.e. $\phi\in \pi_2(\x,\y)$ became a disk $\phi'\in \pi_2(\y,\x)$), but intersection numbers $n_z(\phi),n_w(\phi)$ were unchanged.  Now the relative filtration difference between two intersection points $\x,\y$ can be computed by the equation:
	$$\Filt(\x)-\Filt(\y)=n_z(\phi)-n_w(\phi),$$
\noindent with $\phi\in \pi_2(\x,\y)$  any Whitney disk connecting $\x$ to $\y$. It follows that  	
$$\Filt(\x)-\Filt(\y)=\overline{\Filt}(\y)-\overline{\Filt}(\x),$$ 
\noindent where we temporarily use the notation $\overline{\Filt}$ to indicate the filtration of $\CFa(-Y)$ induced by $K$.  Thus the relative $\Z$-filtration is reversed (changes sign)  upon changing the orientation of $Y$. It follows that $Q_\spinc^*(Y,K,m)$ is isomorphic to $\Filt(-Y,K,m')$ for some $m'$.  It remains to see that $m'=-m-1$.  This would follow if we could show that reversing the orientation of $Y$  reverses the absolute $\Z$-filtration of a generator $\x\in \Ta\cap \Tb$, i.e. $\overline{\Filt(\x)}=-\Filt(\x)$. To this end, recall that the $\Z\oplus\Z$ filtration  of a generator $[\x,i,j] \in \CFKinf(Y,K)$ is given by $(i,j)$ and that 
$$\underline{{\spinc}}(\x) + (i-j)\mathrm{PD}(\mu)=\spinc_0.$$
\noindent Evaluating the Chern class of the $\SpinC$ structure on both sides against $[\widehat{F}]$ yields:
$$\langle \underline{{\spinc}}(\x),[\widehat{F}]\rangle + 2(i-j)=0.$$

Letting $i=0$, it follows that the absolute filtration grading of a generator $\x\in \CFa(Y)$ is given by $j=\OneHalf \langle c_1(\underline{\spinc}(\x)),[\widehat{F}]\rangle$.  This number, in turn, is given by $$\OneHalf \langle c_1(\underline{\spinc}(\x)),[\widehat{F}]\rangle= \OneHalf \EulerMeasure(\PerDom) + n_{\x}(\PerDom),$$
\noindent where $\EulerMeasure(\PerDom)$ is the Euler measure of a periodic domain, $\PerDom$, whose homology class corresponds to $[\widehat{F}]\in H_2(Y_0(K))$, and $n_{\x}(\PerDom)$ is the average of the local multiplicities of $\PerDom$ near the individual intersection points on $\Sigma$ which constitute $\x$  (see Section $7$ of \cite{HolDiskTwo}, specifically Proposition $7.5$, and also Section $2.3$ of \cite{Knots} for further explanation of these terms and the above formula).   Fix a periodic domain $\PerDom$ for $\Sigma$ whose homology class corresponds to $[\widehat{F}]$ and represent it by a map  $$\Phi: (S,\partial S) \rightarrow (\Sigma,\alphas\cup \betas),$$
\noindent where $(S,\partial S)$ is a surface-with-boundary. If we now realize the orientation reversal of $Y$ by reversing the orientation of $\Sigma$, then the map $\Phi$ still gives rise to a periodic domain whose homology class represents $[\widehat{F}]$.  However, the orientation reversal of $\Sigma$  changes the sign of the multiplicities of Im$(\Phi)$. 
It follows that $\OneHalf\EulerMeasure(\PerDom)$ and  $n_{\x}(\PerDom)$ both change sign, and hence $\overline{\Filt}(\x)=-\Filt(\x),$ as claimed.
\end{proof}

\noindent The above propositions show that the pairings:
$$\langle -,-\rangle : \CFa(-Y,\spinc)\otimes \CFa(Y,\spinc)\rightarrow \Z/2\Z,$$
$$\langle -,-\rangle_m : Q_\spinc(-Y,K,\sd m \sd1)\otimes \Filt(Y,K,m)\rightarrow \Z/2\Z,$$

\noindent defined by

	$$\langle \x,\y\rangle = \left\{\begin{array}{ll} 
		1  & \mathrm{if} \ \x=\y\\
		0 & \mathrm{otherwise}\\
	\end{array}\right.$$

\noindent descend to yield  pairings
\begin{eqnarray} \label{eq:pairing}
	\langle -,-\rangle: & \HFa(-Y,\spinc)\otimes \HFa(Y,\spinc)\rightarrow \Z/2\Z \\
\langle -,-\rangle_m: & H_*(Q_\spinc(-Y,K,\sd m \sd1))\otimes H_*(\Filt(Y,K,m))\rightarrow \Z/2\Z.
\end{eqnarray}

\noindent (again, we momentarily suppress $[F]$.)

Thus, given a Floer class $[y]\ne 0\in \HFa(-Y,\spinc)$, a Seifert surface, $F$, and a knot $K$  there are two natural numerical invariants associated to the triple $([y],[F],K)$.   The first invariant is simply $\tau_{[y]}(-Y,[F],K)$ of Definition \ref{defn:tau_x}.  The next uses the filtration $\Filt(Y,[F],K)$ which $K$ induces on $\CFa(Y,\spinc)$.  It measures when the filtration first starts hitting homology classes in $\HFa(Y,\spinc)$ which pair non-trivially with $[y]$.

\begin{defn}
	\label{defn:tau_xdual}
	$$	\tauydualF=\mathrm{min}\{m\in\Z|\ \exists \alpha\in\mathrm{Im} \ I_m \ \mathrm{such \ that \ } \langle [y],\alpha\rangle \ne 0\}. $$
\end{defn}

\begin{remark} Like $\tauxF$, the dependence of $\tauydualF$ on $[F]$ is given by:
	$$\tauydualF-\tau^*_{[y]}(Y,[F'],K)= 	\OneHalf \langle c_1(\spinc),[F\!-\!F'] \rangle.$$
\end{remark}

\noindent {\bf Example: The three-sphere} We conclude this subsection by briefly discussing the case of knots in $S^3$.  In this case, $\HFa(S^3)\cong\Z/2\Z$, supported in grading zero, and we have a canonical Floer homology class given by the generator $\Theta$.  Further, since $-S^3\cong S^3$, we also have $\HFa(-S^3)\cong \Z/2\Z$ (in grading zero) and a canonical generator $\Omega$.  Here we have equality
$\tau_{\Theta}(S^3,K)=\tau^*_{\Omega}(S^3,K)$.   Following \ons \ \cite{FourBall}, we denote this invariant by $\tau(K)$ (this invariant was also defined and studied by Rasmussen \cite{Ras1}).  Since its discovery,  $\tau(K)$ has proved to be rich with geometric content. Indeed, the original motivation for its definition is that $\tau(K)$ is an invariant of the smooth concordance class of $K$ and furthermore provides bounds for the smooth four-ball genus:
$$|\tau(K)|\le g_4(K).$$
Plamenevskaya \cite{Olga} showed that $\tau(K)$ provides bounds on the classical invariants of Legendrian knots in $(S^3,\xi_{std})$ and work of the author \cite{SQPfiber} has shown that $\tau(K)$ detects when a fibered knot bounds a complex curve in the four-dimensional unit ball $B^4\subset \C^2$ of genus equal to the Seifert genus of $K$.

\subsection{Surgery Formula}
\label{subsec:Surgeries}

Let $K\hookrightarrow Y$ be a knot.  A {\em framing} of $K$, denoted $\lambda$, is an isotopy class of simple closed curve on $\partial\nu(K)$ which intersects the meridian disk of $\nu(K)$ once, positively. Let $X_\lambda(K)$ denote the four-manifold obtained by attaching a four-dimensional
two-handle to $[0,1]\times Y$ along $K\hookrightarrow \{1\}\times Y$ with
framing $\lambda$. Note $$\partial X_\lambda(K)=Y\sqcup -Y_\lambda(K)= -Y_\lambda(K)\sqcup -(-Y)$$ \noindent where
 $Y_\lambda(K)$ is the three-manifold obtained by performing
 $\lambda$-framed Dehn surgery on $Y$ along $K$. Thus  $X_\lambda(K)$ can be thought of either as a cobordism from $Y$ to $Y_\lambda(K)$ or as a cobordism from $-Y_{\lambda}(K)$ to $-Y$.  When adopting the latter point of view we denote the cobordism by $\overline{X}_\lambda(K)$.  Given a Spin$^c$ structure
$\spinct$ on $X_\lambda(K)$, there are induced maps $${\widehat
F}_{X_\lambda(K),\spinct}\colon \HFa(Y,\spinct|_Y) \longrightarrow
\HFa(Y_{\lambda}(K),\spinct|_{Y_\lambda(K)}),$$
$${\widehat
F}_{\overline{X}_\lambda(K),\spinct}\colon 
\HFa(-Y_{\lambda}(K),\spinct|_{-Y_\lambda(K)})\longrightarrow \HFa(-Y,\spinct|_{-Y}),$$

\noindent (and also maps for the other versions of Floer homology). These maps are dual to each other under the pairing of Equation \eqref{eq:pairing}: 
\begin{equation}\label{eq:duality}
\langle {\widehat {F}}_{\overline{X}_\lambda(K),\spinct}([x]),[y]\rangle = \langle [x],  {\widehat {F}}_{{X}_\lambda(K),\spinct}([y])\rangle.
\end{equation}
\noindent The maps are induced from corresponding chain maps obtained by counting
pseudo-holomorphic triangles in $\Sym^g(\Sigma)$, as explained in
Section $9$ of \cite{HolDiskTwo}. It was proved in \cite{HolDiskFour} that the maps are invariants of the smooth four-manifold $X_\lambda(K)$ and  $\SpinC$ structures.  Note that \cite{HolDiskFour} assigns maps to arbitrary $\SpinC$ cobordisms, but these will be unnecessary for the present discussion.

Section $4$ of~\cite{Knots},  describes
the relationship between the knot filtration and the \os \ Floer
homologies of three-manifolds obtained by performing ``sufficiently
large'' integral surgeries on $Y$ along $K$.  Moreover, this relationship gives an interpretation of
some of the maps induced by cobordisms in terms of the knot filtration.
These results were generalized to include all rational surgeries on knots in rational homology spheres in \cite{IntegerSurgeries,RationalSurgeries}, but the results of \cite{Knots} will be sufficient for our purposes.  We review these results
here, and refer the reader to \cite{Knots,IntegerSurgeries,RationalSurgeries} for a more thorough
treatment.

Fix a null-homologous knot $K\hookrightarrow Y$, a Spin$^c$ structure, $\spinc\in \SpinC(Y)$, and a Seifert surface, $F$. Framings, $\lambda$, for $K$ are canonically identified with the
integers via the intersection number $\lambda \cm F$ (note that this number is independent of the choice of $F$). 
Further, for a given $n\!>\!0\in\Z$, there are natural
affine identifications $$\SpinC(Y_{\sd n}(K))\cong \SpinC(Y)\times\Zmod{n}$$
$$\SpinC(X_{\sd n}(K))\cong \SpinC(Y)\times\Z$$
where  $Y_{\sd n}(K)$
is the three-manifold obtained by $(\sd n)$-framed surgery on $Y$ along $K$, and  $X_{\sd n}(K)$ is  the associated two-handle cobordism  from $Y$ to $Y_{\sd n}(K)$.  To make these identifications precise, we first fix an orientation of $K$.  This induces an orientation on $F$. The oriented Seifert surface can be capped 
off inside the two-handle to obtain a closed surface ${\widehat
F}$. Now a given $\spinc'\in \SpinC(Y_{\sd n}(K))$  is then identified with a pair $[\spinc,m] \in \SpinC(Y)\times\Zmod{n}$ consisting of a $\SpinC$ structure $\spinc$ which is cobordant to $\spinc'$ via a $\SpinC$ structure, $\spinct_m$ on $X_{\sd n}(K)$ satisfying
\begin{equation}
\label{eq:spinc}
\langle c_1(\spinct_m),[{\widehat F}]\rangle - n = 2m 
\end{equation}
\noindent Furthermore, a $\SpinC$ structure $\spinct_m\in \SpinC(X_{\sd n}(K))$ is uniquely specified by the requirement that $\spinct_m|Y=\spinc$ and that Equation \eqref{eq:spinc} be satisfied.  This yields the latter identification above.  Note, however, that both identifications depend on the homology class $[F]\in H_2(Y;\Z)$.

Theorem $4.1$ of \cite{Knots} shows that for each integer
$m\in\Z$, there is an integer $N$ so that for all $n\geq N$,
we have the isomorphism:
$$ H_*(C_\spinc\{\min(i,j- m)=0\})\cong \HFa(Y_{\sd n}(K),[\spinc,m]).$$

\noindent There is a natural chain map $$f_m: C_\spinc\{i=0\}\longrightarrow
C_\spinc\{\min(i,j- m)=0\},$$ \noindent which is defined as the inclusion on the quotient complex $C_\spinc\{i=0,j\ge m\}$ and is zero for the subcomplex
$C_\spinc\{i=0,j<m\}$. The
proof of Theorem $4.1$ of \cite{Knots}
shows that $f_m$ induces the map  $$\widehat{F}_{X_{\sd n}(K),_{\spinct_m}}:\HFa(Y,\spinc)\longrightarrow \HFa(Y_{\sd n}(K),[\spinc,m]),$$ \noindent given by the two-handle addition, endowed with the unique
$\SpinC$ structure $\spinct_m$ restricting to $\spinc$ on $Y$ and satisfying Equation \eqref{eq:spinc} above
(again, provided that $n$ is sufficiently large compared to $m$ and the
genus of the knot).  Note that in order for this theorem to be used as stated, the labeling of $\SpinC$ structures on $Y_{\sd n}(K)$ and $X_{\sd n}(K)$ must be induced by the same homology class of Seifert surface as used in the definition of $\CFKinf(Y,[F],K,\spinc)$.   Finally, we remark that Theorem $4.1$ is stronger than what we have stated.  It identifies the various \os \ homologies, $\HFp(Y_{\sd n}(K),[\spinc,m]), \HFm(Y_{\sd n}(K),[\spinc,m]),$ and $\HFinf(Y_{\sd n}(K),[\spinc,m])$ with the homology of certain sub and quotient complexes of $\CFKinf(Y,[F],K,\spinc)$.  This level of generality, however, will not be necessary for our purposes.

\subsection{Background on the \os \ contact invariant}

\label{subsec:contact}
In this subsection we briefly review the definition and basic properties of the \os \ contact invariant. 
A fundamental theorem in three-dimensional contact geometry, due to Giroux \cite{Giroux}, states that the construction of Thurston and Winkelnkemper \cite{Thurston} discussed in the introduction can be reversed.  Moreover, Giroux's theorem states that there is an equivalence:
$$\frac{\{\mathrm{open \ book \ decompositions \  of\ Y^3}\}}{\{\mathrm{positive\ Hopf\ stabilization}\}} \simeq \frac{\{\mathrm{contact\ structures\ on\ Y^3}\}}{\{\mathrm{isotopy}\}}$$
\noindent See \cite{Etnyre2} for an exposition of this theorem.

Thus, associated to a contact structure is an equivalence class of open book decompositions of $Y$, where any two open books are related by a sequence of plumbing and deplumbing of positive Hopf bands.  

Choose then a fibered knot $(F,K)$ whose associated open book decomposition supports the contact structure $(Y,\xi)$.  In \cite{Contact} \ons \ show that the knot Floer homology of a fibered knot satisfies:
$$H_*({\mathcal{F}}_{\spinc_\xi}(\sd Y,[F],K,\sd g(F)))\cong \Z/2\Z,$$

 \noindent where $\spinc_{\xi}$ is the $\SpinC$ structure on $Y$ associated to the contact structure, $\xi$. They further show that this group is generated by a homogeneous cycle supported in grading equal to the Hopf invariant of the two-plane field of $\xi$.   Let $c_0$ denote a generator of this group. We define:
 $$c(F,K) = \sd I_{\sd g(F)}(c_0)\in \HFa(\sd Y,\spinc_{\xi}),$$
\noindent  where $\sd I_{\sd g(F)}$, as above, is the map on homology induced by the inclusion  
$$ \sd \iota_{\sd g(F)}: \mathcal{F}_{\spinc_\xi}(\sd Y,[F],K,\sd g(F))\longrightarrow \CFa(\sd Y,\spinc_{\xi}).$$
\noindent \ons \ showed that $c(F,K)$ depends only on the contact structure induced by the open book decomposition associated to $(F,K)$ (Theorem  $1.3$ of \cite{Contact}).  Thus we have the  \os \ contact invariant: 
$$c(\xi):= c(F,K),$$
\noindent where $(F,K)$ is any fibered knot whose open book supports $\xi$.

 The contact invariant enjoys the following properties:
\begin{enumerate}
	\item (Vanishing \cite{Contact})  If $\xi$ is overtwisted then $c(\xi)=0$
	\item (Non-Vanishing \cite{GenusBounds}) If $(W,\omega)$ is a strong  symplectic filling of $(Y,\xi)$ then $c(\xi)\ne 0$.
	\item (Naturality \cite{Contact,Lisca,Olga2,Ghiggini}) If $(\sympW,\omega)$ denotes the symplectic cobordism between $(Y,\xi)$ and $(Y_K, \xi_K)$ induced by Legendrian surgery along a Legendrian knot $K$,	then we have $$\widehat{F}_{\sympWbar,{\spinck}}(c(\xi_K))=c(\xi),$$ where $\spinck$ is the canonical 
		$\SpinC$ structure induced by the symplectic form $\omega$.  Furthermore, if $\spinct\ne \spinck$ we have $$\widehat{F}_{\sympWbar,\spinct}(c(\xi_K))=0$$
	\item (Product Formula) Let $(Y_1\#Y_2,\xi_1\#\xi_2)$ denote the contact connected sum of $(Y_1,\xi_1)$ and $(Y_2,\xi_2)$ (see \cite{Etnyre1}).  Theorem $6.1$ of \cite{HolDiskTwo} indicates that there is an isomorphism $$\HFa(\sd Y_1\#\sd Y_2,\spinc_{\xi_1}\#\spinc_{\xi_2})\cong \HFa(\sd Y_1,\spinc_{\xi_1})\otimes_{\Z/2\Z} \HFa(\sd Y_2,\spinc_{\xi_2}).$$  Under this isomorphism, $c(Y_1\#Y_2,\xi_1\#\xi_2)=c(Y_1,\xi_1)\otimes c(Y_2,\xi_2)$. 
\end{enumerate}
\bigskip

\noindent We expound upon Properties $3$ and $4$.  To understand Property $3$, first recall that to a contact three-manifold $(Y,\xi)$ with a Legendrian knot, $K$, Weinstein \cite{Weinstein} constructs a symplectic cobordism  $(\sympW,\omega)$ between $(Y,\xi)$ and a contact manifold $(Y_K,\xi_K)$.  Topologically, $Y_K$ is the manifold obtained by  $(tb\!-\!1)$-framed surgery along $K$ and $W_K$ is the corresponding two-handle cobordism,  where $tb$ is the Thurston-Bennequin number of $K$.  The contact structure $\xi_K$ is constructed so that it agrees with $\xi$ on $Y\!-\!\nu(K)$.  Gompf shows (Proposition $2.3$ of \cite{Gompf}) that the first Chern class of the canonical $\SpinC$ structure, $\spinck$ of $(\sympW,\omega)$ satisfies:
$$\langle c_1(\spinck),[\widehat{F}]\rangle = rot_F(K),$$
\noindent and it is clear that $$tb(K)- 1= [\widehat{F}]\cm[\widehat{F}]$$
\noindent where $F$ is the Seifert surface for $K$ defining the $0$-framing, and $\widehat{F}$ is the closed surface in the cobordism obtained by capping off $F$ with the core of the two-handle. Now we have $$\partial \sympWbar = (- Y_K)\sqcup - (- Y).$$
\noindent  As a cobordism from $- Y_K$ to $- Y$, $\sympWbar$ induces maps on Floer homology as described in the preceding subsection, and the naturality statement says that the contact invariants behave nicely in the presence of the symplectic structure on $\sympW$.   We should  mention that the naturality property was proved in increasing levels of generality by \ons \ (Theorem $4.2$ of \cite{Contact}), Lisca and Stipsicz (Theorem $2.3$ of \cite{Lisca}) and Ghiggini (Proposition 3.3 of \cite{Ghiggini}).  In fact, Lisca and Stipsicz's result is a naturality statement for the contact invariant under contact $+1$ surgery i.e. $tb+1$ framed surgery, and does not make mention of the $\SpinC$ structure on $W_{tb+1}(K)$, but instead sums over all $\SpinC$ structures.  The statement we have included  as Property $3$ is nearly identical to Ghiggini's result, but here we have stated the result for the maps induced on the ``hat'' version of \os \ homology.  Ghiggini's result is for an analogous contact invariant $c^+(Y,\xi)\in \HFp(-Y,\spinc_\xi)$ and for the map on $\HFp$ induced by Weinstein's cobordism.  As stated, Property $3$ follows easily from Ghiggini's result and naturality of the long exact sequence relating $\HFp$ to $\HFa$ (Lemma $4.4$ of \cite{HolDisk}) with respect to maps induced by cobordisms:

To the best of our knowledge, a proof of Property $4$ does not exist in the literature but is straightforward.  For completeness, we spell out the details here.
\bigskip

\noindent {\bf Proof of Property $4$. } Let $K_1\hookrightarrow Y_1$ and $K_2\hookrightarrow Y_2$ be fibered knots equipped with fiber surfaces $F_1$ and $F_2$ whose associated open book decompositions induce $(Y_1,\xi_1)$ and $(Y_2,\xi_2)$, respectively.  Then the connected sum $K_1\#K_2 \hookrightarrow Y_1\#Y_2$  is a fibered knot equipped with fiber surface $F_1\natural F_2$, where $\natural$ denotes boundary connected sum.   Torisu \cite{Torisu} shows that the contact structure associated to the resulting open book decomposition of $Y_1\#Y_2$ is isotopic to $(Y_1\#Y_2,\xi_1\#\xi_2)$.  As for the Floer homology of $Y_1\#Y_2$,  \ons \ proved (Proposition $6.1$ of \cite{HolDiskTwo}) that there is an isomorphism $$\HFa(\sd Y_1\#\sd Y_2,\spinc_{\xi_1}\#\spinc_{\xi_2})\cong \HFa(\sd Y_1,\spinc_{\xi_1})\otimes_{\Z/2\Z} \HFa(\sd Y_2,\spinc_{\xi_2}),$$\noindent  induced by a chain homotopy equivalence \begin{equation}\label{eq:tensor} \CFa(\sd Y_1\#\sd Y_2,\spinc_{\xi_1}\#\spinc_{\xi_2})\cong \CFa(\sd Y_1,\spinc_{\xi_1})\otimes_{\Z/2\Z} \CFa(\sd Y_2,\spinc_{\xi_2}).\end{equation}
	\noindent Theorem $7.1$ of \cite{Knots} states that an analogous result holds in the category of filtered chain complexes when we form the connected sum of knots.  More precisely, recall from Subsection \ref{subsec:Filtration} that associated to $K_j$ we have a $\Z$-filtration of $\CFa(\sd Y_j,\spinc_j)$, $j=1,2$.  We denoted the subcomplexes of this filtration by $\mathcal{F}_{\spinc_j}(\sd Y_j,K_j,m)$, so that there are inclusions:
$$ \sd\iota_m^{K_j}: \mathcal{F}_{\spinc_j}(\sd Y_j,K_j,m)\hookrightarrow \CFa(\sd Y_j,\spinc_j)$$

The inclusion maps $\sd\iota_m^{K_j}$ induce a filtration of $\CFa(\sd Y_1,\spinc_1)\otimes_{\Z/2\Z}
\CFa(\sd Y_2,\spinc_2)$ as the image of $$
\sum_{m_1+m_2=m}\sd\iota_{m_1}^{K_1}\otimes \sd\iota_{m_2}^{K_2}\colon \!\!\!
\bigoplus_{m_1+m_2=m} \!\!\!\mathcal{F}_{\spinc_1}(\sd Y_1,K_1,m_1)\otimes\mathcal{F}_{\spinc_2}(\sd Y_2,K_2,m_2) \longrightarrow \CFa(\sd Y_1,\spinc_1)\otimes \CFa(\sd Y_2,\spinc_2).$$ 

According to Theorem $7.1$
of \cite{Knots}, under the chain homotopy equivalence given by Equation \eqref{eq:tensor}, the above filtration of $\CFa(\sd Y_1,\spinc_1)\otimes_{\Z/2\Z}
\CFa(\sd Y_2,\spinc_2)$ is
identified with the filtration of $\CFa(\sd Y_1\#\sd Y_2,\spinc_1\#\spinc_2)$ induced by the connected
sum $K_1\# K_2$.

It follows immediately that 
$$\sd\iota_{\sd g(F_1\natural F_2)}^{K_1\#K_2}(c_0(K_1\#K_2))= \sd\iota_{\sd g(F_1)}^{K_1}(c_0(K_1))\otimes \sd\iota_{\sd g(F_2)}^{K_2}(c_0(K_2)),$$
\noindent where $c_0(K_1\#K_2)$, $c_0(K_1)$, and $c_0(K_2)$ are cycles whose homology classes generate $$H_*(\mathcal{F}_{\spinc_1\# \spinc_2}(\sd Y_1\#\sd Y_2,K_1\#K_2,\sd g(F_1\natural F_2))\cong  $$
$$\ \ \ \ \ H_*(\mathcal{F}_{\spinc_1}(\sd Y_1,K_1,\sd g(F_1))\cong  H_*(\mathcal{F}_{\spinc_2}(\sd Y_2, K_2,\sd g(F_1))\cong \Z/2\Z,$$ \noindent respectively.   Property $4$ now follows from the definition of the contact invariant. $\square$

\subsection{Definition of $\tauc$}
With all necessary background in place, we can define the invariant which will be our main object of study:
\begin{defn}
	\label{defn:tauc} Let $K\hookrightarrow Y$ be a knot, $F$ a Seifert surface for $K$, and $\xi$ a contact structure with $c(\xi)\ne0$, then
	$$\taucF := \tau^*_{c(\xi)}(Y,[F],K),$$
	\noindent where the right-hand side is the invariant of Definition \ref{defn:tau_xdual}.
\end{defn}

Note that our notation suppresses the $\SpinC$ structure on $Y$, but that it is implicitly specified, since $c(\xi) \in \HFa(-Y,\spinc_{\xi})$. Thus $\taucF$ is defined via the filtration of $\CFa(Y,\spinc_{\xi})$.  It is immediate from the theorems of \ons \ concerning the invariance of $\HFa(Y,\spinc)$, ${{\Filt}}(Y,[F],K,m)$, and $c(\xi)$, that $\taucF$ is an invariant of the quadruple $(Y,[F],K,\xi)$ (the invariance theorems alluded to are Theorem $1.1$ of \cite{HolDisk}, Theorem $3.1$ of \cite{Knots}, and Theorem $1.3$ of \cite{Contact}, respectively).

\section{Properties of $\tau_{[x]}(Y,K)$ and $\tau^*_{[y]}(Y,K)$}
\label{sec:Properties}
Fix non-vanishing Floer classes $[x]\in \HFa(Y)$ and  $[y]\in \HFa(-Y)$. In this section we prove some basic properties of $\taux$,  $\tauydual$, (see Definitions \ref{defn:tau_x}, \ref{defn:tau_xdual}). Throughout, we will suppress the Seifert surface from the notation as much as possible, calling attention to its role when  when there may be ambiguity.  The properties here generalize properties of the \os \ concordance invariant, $\tau(K)$, most of which are established in Section $3$ of \cite{FourBall}. For the present paper we will be primarily interested in the case when $[y]=\contact$, and indeed the main theorems will utilize $\tauc:= \tau^*_{\contact}(Y,K)$. We choose to discuss the more general invariants for arbitrary non-zero classes since they also contain geometric content, see for instance \cite{Grigsby}.  It will thus be useful to collect in one place the general algebraic properties of these invariants.  The following section will use the properties developed here to prove the theorems stated in the introduction.

Let $W_{\!\sd n}(K)$ be the cobordism from from $Y$ to $Y_{\sd n}(K)$ induced from the two-handle attachment along $K\hookrightarrow Y$ with framing $-n\!<\!0$.  Subsection \ref{subsec:Surgeries} indicates that associated to each $\spinct\in \SpinC(W_{\!\sd n}(K))$ there is a map: $${\widehat F}_{W_{\!\sd n}(K),\spinct}: \HFa(Y,\spinct|_Y)\longrightarrow \HFa(Y_{\!\sd n}(K),\spinct|_{Y_{\!\sd n}(K)}),$$
\noindent  Fix $\spinc\in \SpinC(Y)$.  To simplify notation, we use ${\widehat F}_{\!\sd n,m}$ to denote the map ${\widehat F}_{W_{\!\sd n}(K),\spinct_m}$ associated to  the unique $\spinct_m\in \SpinC(W_{\!\sd n}(K))$  satisfying:
\begin{itemize}
	\item $\spinct_m|Y=\spinc$
	\item $\langle c_1(\spinct_m),[{\widehat F}]\rangle - n = 2m,$
\end{itemize} where, by an abuse of notation, $[{\widehat F}]$ denotes (in addition to the map of Floer homology) the homology class of a fixed Seifert surface, $F$, capped off in the two-handle to yield a closed surface, ${\widehat F}$. In light of the relationship between the knot Floer homology filtration and the maps on Floer homology induced by four-dimensional two-handle attachment, we have the following proposition.  Roughly speaking, it says that $\taux$ controls when $\Mapn$ maps $[x]$ non-trivially.

\begin{prop}
	\label{prop:FourDInterp} (Compare Proposition $3.1$  of \cite{FourBall}) Let $[x]\!\ne\!0\in \HFa(Y,\spinc)$ be a non-trivial Floer homology class and let $n>0$ be sufficiently large. We have
	\begin{itemize} 
		\item 	If $m<\taux$, then $\Mapn([x])\ne 0$
 		\item  If $m>\taux$, then $\Mapn([x])=0$
	\end{itemize}
\end{prop}
\begin{remark} Changing $[F]$ changes $\taux$ according to Remark \ref{remark:dependence}.  However, changing $[F]$ also changes the labeling of $\spinct_m\in \SpinC(W_{\!\sd n}(K))$ according to Equation \eqref{eq:spinc}, and the two changes cancel. 
\end{remark}
\begin{proof}
Consider the following commutative diagram of chain complexes and chain maps,
\begin{equation}
\label{eq:Diag1}
\begin{CD}
	0\to \filtY{m} @>{\iota_m}>> C_\spinc\{i=0\}\simeq\CFa(Y,\spinc) @>{p_m}>> Q_\spinc(Y,K,m)\to0\\
 @V{\Pi}VV @V{f_m}VV @V{\cong}VV \\ 
0\to C_\spinc\{i\geq 0,j=m\} @>>> C_\spinc\{\min(i,j-m)=0\} @>>>Q_\spinc(Y,K,m) \to0,
\end{CD}
\end{equation}
   The vertical map on the left is defined to vanish on the subcomplex $ \Cij{=0}{\le m-1}$ of $\Filt(Y,K,m)=\Cij{=0}{\le m}$, while the vertical map on the right is simply the identity.  The middle vertical map is the chain map described in Subsection \ref{subsec:Surgeries}, which vanishes on the subcomplex, $\Cij{=0}{< m}$ of $\Ci{=0}$.  Theorem $4.1$ of \cite{Knots} states that
for $n$ sufficiently large, we have an identification
$$C_\spinc\{\min(i,j-m)=0\}\simeq\CFa(Y_{\!\sd n}(K),[\spinc,m])$$ under which the map
$f_m$ represents the chain map inducing ${\widehat F}_{\sd n,m}$ above.  Let $I_m$, $P_m$ denote the maps on homology induced by $\iota_m,p_m$. Now, if
$m<\tau_{[x]}(Y,K)$, we have that $P_m([x])\ne 0$ (by the long exact sequence associated to the upper short exact sequence)  and hence ${\widehat F}_{\sd n,m}([x])\ne 0$. 
Moreover, since  $f_m$ is trivial on $$C_\spinc\{i=0,j\leq
m-1\}=\filtY{m\!-\!1},$$ it factors through the map $p_{m-1}$.  Thus the map on homology, ${\widehat F}_{\sd n,m}$, factors through $P_{m-1}$.  If $m>\taux$, then $P_{m-1}([x])=0$ (again by the upper long exact sequence), and hence ${\widehat F}_{\sd n,m}([x])=0$.
\end{proof}

  Similarly, for $[y]\in \HFa(-Y)$ the dual invariant $\tauydual$ controls how $\Mapn$ maps Floer classes $\alpha \in \HFa(Y)$  which pair non-trivially with $[y]$:

\begin{prop}
	\label{prop:FourDInterpdual}  Let $[y]\!\ne\!0\in \HFa(-Y,\spinc)$ be a non-trivial Floer homology class and let $n>0$ be sufficiently large. Then we have
	\begin{itemize}
		\item 	If $m<\tauydual$, then for every $\alpha\in \HFa(Y,\spinc)$ such that  $\langle [y],\alpha\rangle \ne 0$ \newline $\Mapn(\alpha)\ne 0$
		\item If $m>\tauydual$, then there exists  $\alpha\in \HFa(Y,\spinc)$ such that $\langle [y],\alpha\rangle \ne 0$ and $\Mapn(\alpha)=0$.
	\end{itemize}

\end{prop}
\begin{remark} Here $\langle -,-\rangle: \HFa(-Y)\otimes\HFa(Y)\rightarrow \Z/2\Z$ is the pairing (Equation \eqref{eq:pairing}) defined in Subsection \ref{subsec:Filtration}. \end{remark}
\begin{proof}The proof is the sames as the preceding Proposition, bearing in mind the definition of $\tauydual$.
\end{proof}

Given a non-vanishing Floer class $[y]\in \HFa(-Y,\spinc)$ we can use the filtrations $\Filt(-Y,K)$ and $\Filt(Y,K)$ to define  $\tauy$ or $\tauydual$, respectively.    The next proposition says that the two invariants are related by a change of sign.

\begin{prop} 
	\label{prop:reversal} (Compare Proposition $3.3$ of \cite{FourBall})
	Let $[y]\ne0\in \HFa(-Y,\spinc)$.  Then 
	$$\tauy=-\tauydual$$
\end{prop}

\begin{proof} Proposition \ref{prop:knotduality} states that the short exact sequence corresponding to $\Filt(\sd Y,K,\sd m\sd 1)$ is naturally isomorphic to the dual of the short exact sequence coming from $\Filt(Y,K,m)$.  Specifically, recall Commutative Diagram \eqref{eq:knotduality}:
$$
	\begin{CD}
	0 \leftarrow \Filt^*(Y,K,m) @<{\iota_m^*}<< \CFa^*(Y,\spinc) @<{p_m^*}<< Q_\spinc^*(Y,K,m) \leftarrow 0 \\
  @V{\cong}VV @V{\cong}VV @V{\cong}VV  \\ 
 	0\leftarrow  Q_\spinc(\sd Y,K,\sd m\sd 1) @<{\sd p_{\sd m\sd 1}}<< \CFa(\sd Y,\spinc) @<{\sd \iota_{\sd m\sd 1}}<< \Filt(\sd Y,K,\sd m\sd 1) \leftarrow 0 \\
\end{CD}
$$	
	Thus the inclusion and projection maps $-\iota_{-m-1}$ and $-p_{-m-1}$  are identified with the dual maps $p_m^*$ and $\iota_m^*$, respectively.  If follows that the induced maps $\sd P_{\sd m \sd 1}$ and $I_m$ are adjoint with respect to the pairings of Equation \eqref{eq:pairing} and $(6)$ i.e. for any $[x]\in H_*(\Filt(Y,K,m))$ and $[y]\in \HFa(\sd Y)$ we have 
$$ \langle \sd P_{\sd m\sd 1}([y]),[x]\rangle_m = \langle [y], I_m ([x])\rangle.$$ \noindent  Suppose that $ \tauydual=m$. The definition then implies that there exists $\alpha=I_m(a)$ such that $$0\ne \langle [y],\alpha\rangle = \langle [y],I_m(a)\rangle = \langle \sd P_{\sd m \sd 1}([y]),a\rangle.$$  Thus $\sd P_{\sd m \sd 1}([y])\ne 0$.  This implies $[y]\notin \mathrm{Im}(\sd I_{\sd m \sd 1})$, by the long exact sequence coming from the lower short exact sequence in the above commutative diagram.  Hence $$\tauy\ge - m=- \tauydual.$$  We wish to show that the inequality is in fact an equality.  Assume, then, that 
$$\tauy=k>- \tauydual.$$   Then $[y]\notin \mathrm{Im}(-I_{k\sd 1})$ and hence $$0\ne \sd P_{k\sd 1}([y])\in H_*(Q_\spinc(\sd Y,K,k\sd 1))\cong H^*(\Filt^*(Y,K,\sd k)).$$  Thus there exists $a\in H_*(\Filt(Y,K,\sd k))$ such that    $$0\ne \langle \sd P_{k\sd 1}[y],a\rangle = \langle [y],I_{\sd k}(a)\rangle,$$ \noindent  It follows that $\tauydual \le - k$, contradicting the assumption.  

\end{proof}

Similar to the  \os \ concordance invariant, both $\taux$ and $\tauydual$ satisfy an additivity property under connected sums.  As in \cite{FourBall}, this  follows readily from 
Theorem $7.1$ of \cite{Knots}, which explains the behavior of the knot Floer homology filtration under the connected sum of knots.

\begin{prop}
	\label{prop:Additivity} (Compare Proposition $3.2$ of \cite{FourBall})
Let $K_1$ and $K_2$ be knots in three-manifolds $Y_1$ and $Y_2$, respectively and let $K_1\# K_2$
denote their connected sum. Then, for any pair of non-vanishing Floer classes $[x_i]\in \HFa(Y_i,\spinc_i)$, we have 
$$\tau_{[x_1]\otimes [x_2]}(Y_1\#Y_2, K_1\# K_2)=\tau_{[x_1]}(Y_1,K_1)+\tau_{[x_2]}(Y_2,K_2),$$
\noindent where $[x_1]\otimes [x_2]$ denotes the image of $[x_1]$ and $[x_2]$ under the isomorphism
$$\HFa(Y_1,\spinc_{1})\otimes_{\Z/2\Z} \HFa(Y_2,\spinc_{2}) \cong \HFa(Y_1\# Y_2,\spinc_{1}\#\spinc_{2}).$$
Similarly, for $\tau^*$ we have
$$\tau^*_{[y_1]\otimes [y_2]}(Y_1\#Y_2, K_1\# K_2)=\tau^*_{[y_1]}(Y_1,K_1)+\tau^*_{[y_2]}(Y_2,K_2),$$
\noindent for any non-vanishing $[y_i]\in \HFa(-Y_i,\spinc_i)$.

\end{prop}
\begin{remark} The Seifert surface used for $K_1\#K_2$ should, in each case, be the boundary connected sum, $F_1\natural F_2$, of the Seifert surfaces $F_1$ and $F_2$ used for $K_1$ and $K_2$, respectively. 
\end{remark}
\begin{proof} According to Theorem $7.1$ of \cite{Knots}, the filtration of $\CFa(Y_1\#Y_2,\spinc_1\#\spinc_2)$ induced by $K_1\# K_2$ is filtered chain homotopy equivalent to the filtration of $\CFa(Y_1,\spinc_{\xi_1})\otimes_{\Z/2\Z} \CFa(Y_2,\spinc_{\xi_2})$ induced by the tensor product of inclusion maps for $K_1,K_2$:

$$
\begin{CD}
	\underset{m_1+m_2=m}{\sum} \iota_{m_1}^{K_1}\otimes \iota_{m_2}^{K_2}\colon \!\!\!
	\underset{m_1+m_2=m}{\bigoplus} \!\!\!\mathcal{F}_{\spinc_1}( Y_1,K_1,m_1)\otimes\mathcal{F}_{\spinc_2}(Y_2,K_2,m_2) @>>> \CFa( Y_1,\spinc_1)\otimes \CFa( Y_2,\spinc_2) \\
@V{\cong}VV @V{\cong}VV \\
\iota_m^{K_1\#K_2}\ \ \ \colon  \ \ \ \ \ \mathcal{F}_{\spinc_1\#\spinc_2}(Y_1\#Y_2,K_1\# K_2,m) @>>> \CFa(Y_1\#Y_2,\spinc_1\#\spinc_2).
\end{CD} 
$$	
\noindent It follows that  
Im$(I^{K_1\#K_2}_m)$ contains $[x_1]\otimes [x_2]$  if and only if there
is a decomposition $m=m_1+m_2$ such that Im$(I^{K_1}_{m_1})$ and  Im$(I^{K_2}_{m_2})$
contain $[x_1]$ and $[x_2]$, respectively.
The minimum value of $m$ for which this occurs is clearly $m=\tau_{[x_1]}(Y_1,K_1)+\tau_{[x_2]}(Y_2,K_2)$.  Additivityof $\tau^*_{[y]}$ follows from the additivity of $\tau_{[x]}$ just proved, and the preceding proposition.  Indeed, we have:
$$-\tau^*_{[y_1]\otimes [y_2]}(Y_1\#Y_2, K_1\# K_2)= \tau_{[y_1]\otimes [y_2]}(\sd Y_1\#\sd Y_2, K_1\# K_2)= $$ 
$$  \ \ \ \ =\tau_{[y_1]}(\sd Y_1,K_1)+\tau_{[y_2]}(\sd Y_2,K_2)=-\tau^*_{[y_1]}(Y_1,K_1)-\tau^*_{[y_2]}(Y_2,K_2),$$
\noindent where the first and last equalities follow from  Proposition \ref{prop:reversal}, and the middle equality from the first part of the present proposition applied to the manifolds $-Y_1,-Y_2$.  
\end{proof}

\section{Proof of Theorems}
\label{sec:proofs}
We now apply the  general properties of $\tauy$ and $\tauydual$ established in the previous section to the special case when $[y]=c(\xi)\in \HFa(-Y,\spinc_\xi)$.  Recall that in this case we have denoted $\tau^*_{\contact}(Y,K)$ by $\tauc$.  Throughout this section, we will fix the homology class of Seifert surface once and for all, so that any invariant or identification depending on this choice will use the same $[F]\in H_2(Y;\Z)$.  In light of this, we will simplify the notation by omitting $[F]$ whenever possible.    

\bigskip

\noindent {\bf Proof of Theorem \ref{thm:tbbounds}:} We follow Plamenevskaya's proof \cite{Olga} of the analogous theorem for $K\hookrightarrow (S^3,\xi_{std})$ and $\tau_{\xi_{std}}(S^3,K)=\tau(K)$.  
Assume that we have a Legendrian representative $\tilde{K}$ with Thurston-Bennequin and rotation numbers $tb(\tilde{K})$ and $rot(\tilde{K})$, respectively.   Since changing the orientation of the knot changes the sign of
its rotation number, it suffices to prove the inequality
$$ tb(\tilde{K}) + rot(\tilde{K}) \le 2\tauc - 1,  $$
\noindent for any oriented Legendrian knot. This is because $\tauc$, and indeed the $\Z\oplus \Z$ filtered chain homotopy type of $\CFKinf(Y,K)$, is independent of the orientation on $K$ (see Proposition $3.8$ of \cite{Knots}).   
According to \cite{Weinstein,Gompf} there is a symplectic cobordism $(\sympW,\omega)$ between $(Y,\xi)$ and $(Y_K,\xi_K)$ induced by Legendrian surgery along $\tilde{K}$ which satisfies:
$$\langle c_1(\spinck),[\widehat{F}]\rangle = rot_F(\tilde{K}),$$
$$ [\widehat{F}]\cm[\widehat{F}]=tb(\tilde{K})-1,$$
 \noindent where $\spinck$ is the canonical $\SpinC$ structure associated to $(\sympW,\omega)$. The naturality property of the contact invariant (Property $3$ in Subsection \ref{subsec:contact}) indicates that 
 \begin{equation}
	 \label{eq:naturality}
	\widehat{F}_{\overline{W}_K,{\spinck}}(c(\xi_K))=c(\xi)
\end{equation}
\noindent 

Pick any homogeneous $\alpha\in \HFa(Y,\spinc_\xi)$ which pairs non-trivially with $\contact\in \HFa(-Y,\spinc_\xi)$ under Equation \eqref{eq:pairing}.  It follows from Equations \eqref{eq:naturality} and \eqref{eq:duality} that:
$$0\ne \langle \contact,\alpha \rangle = \langle \widehat{F}_{\overline{W}_K,\spinck}(c(\xi_K)), \alpha\rangle = \langle c(\xi_K), \widehat{F}_{\sympW,\spinck}(\alpha)\rangle,$$
\noindent and hence that $\widehat{F}_{\sympW,\spinck}(\alpha)\ne 0$.
	 
Thus every homogeneous class pairing non-trivially with $c(\xi)$ is mapped non-trivially by $\widehat{F}_{\sympW,\spinck}$, and so we wish to use Proposition \ref{prop:FourDInterpdual} to bound $\langle c_1(\spinck),[\widehat{F}]\rangle+ [\widehat{F}]^2$   in terms of $\tauc$.  To carry this out, we stabilize $\tilde{K}$ (i.e. add kinks to the front projection of a portion of $\tilde{K}$ contained in a Darboux neighborhood) to decrease the Thurston-Bennequin number and increase the rotation number while keeping $tb + rot_F$ constant.  Thus, without loss of generality we may  assume that the framing $tb-1=-n$ for the Legendrian surgery is sufficiently negative for Proposition \ref{prop:FourDInterpdual} to hold. This immediately yields 
$$ rot_F(\tilde{K})+ tb(\tilde{K})-1 = \langle c_1(\spinck),[{\widehat F}]\rangle - n = 2m \le 2\tauc.$$
\noindent  We have shown that $ tb(\tilde{K}) + rot_F(\tilde{K}) \le 2\tauc +1.$ To achieve the stated inequality, we examine the connected sum $K\# K\hookrightarrow (Y\#Y,\xi\#\xi)$.  It is straightforward to see that that:
\begin{equation}\label{eq:connectedtb}
\underset{\widetilde{K\# K}}{\mathrm{max}}\ [ tb(\widetilde{K\# K}) + rot_{F\natural F}(\widetilde{K\# K}) ]\ge  2\ \underset{\tilde{K}}{\mathrm{max}}\ [ tb(\tilde{K})+ rot_F(\tilde{K})]+1,
\end{equation}
\noindent where the maximum on both sides is taken over all Legendrian representatives of $K\#K$ and $K$, respectively. Indeed, to see the inequality, simply take the connected sum of a particular representative, $\tilde{K}$, of $K$ maximizing $tb(\tilde{K})+rot_F(\tilde{K})$.  Under this sum, $tb(\tilde{K}\#\tilde{K})=2tb(\tilde{K})+1$ and $rot_{F\natural F}(\tilde{K}\#\tilde{K})=2rot_F(\tilde{K})$, establishing the inequality (in fact, work of Etnyre and Honda \cite{Etnyre3} shows that equality is always satisfied in \eqref{eq:connectedtb}).   Now Property $4$ of the contact invariant in Subsection \ref{subsec:contact} states that $c(\xi\#\xi)=c(\xi)\otimes c(\xi)$.  It then follows from the additivity of $\tauyup$ under connected sums (Proposition \ref{prop:Additivity}) that $\tau_{\xi\#\xi}(K\# K) = 2\tau_\xi(K)$.  Combining this with inequality \eqref{eq:connectedtb}, we have
$$ 2(tb(\tilde{K})+rot_F(\tilde{K}))+1 \le \underset{\widetilde{K\# K}}{\mathrm{max}}\ [tb(\tilde{K\# K}) + rot_{F\natural F}(\tilde{K\# K})]\le $$ $$ \ \ \ \ \ \ \ \ \ \ 2\tau_{\xi\#\xi}(K\# K)+1 =4\tauc +1 ,$$
\noindent Where $\tilde{K}$ is any Legendrian representative of $K$. In other words, we have
$$ tb(\tilde{K})+rot_F(\tilde{K}) \le 2\tauc$$
\noindent The theorem follows from the observation that $tb+rot$ is always odd since 
$$rot_F(K)= \langle c_1(\spinck),[\widehat{F}]\rangle = [\widehat{F}]\cm [\widehat{F}] = tb(K)-1 \mod2,$$
\noindent which in turn follows from the fact that $c_1(\spinck)$ is characteristic.

$\square$

\bigskip
\noindent {\bf Proof of Theorem \ref{thm:fibered}:} The Theorem will follow from Theorem \ref{thm:tbbounds} and Proposition \ref{prop:reversal}, together with the  
definitions of the contact invariant and $\tauc$.  Assume we have a Legendrian realization $\tilde{K}$ of the fibered knot $(F,K)$ for which:
	 $$tb(\tilde{K})+|rot_F(\tilde{K})|= 2g(F)-1.$$
	 By Theorem  \ref{thm:tbbounds} we have that $tb(\tilde{K})+|rot_F(\tilde{K})|\le 2\tauc-1$.  However, the adjunction inequality (Theorem $5.1$ of \cite{Knots}) states that $$\HFKa(Y,[F],K,m)=0 \ \ \mathrm{if\ } |m| > g(F).$$
	 Since	$H_*(\frac{\filtY{m}}{\filtY{m-1}}):= \HFKa(Y,[F],K,m)$, it  follows that $$H_*(\filtY{m})\cong H_*(\filtY{m-1}) \ \ \mathrm{if \ } m > g(F).$$  Thus $H_*(\filtY{g(F)})\cong \HFa(Y,\spinc)$ implying that $\tauc\le g(F)$.

	Summarizing, we have:
	$$2g(F)-1 = tb(\tilde{K})+|rot_F(\tilde{K})|\le 2\tauc-1 \le 2g(F)-1.$$
	\noindent	Thus $\tauc=g(F)$.  Now Proposition \ref{prop:reversal} tells us that
	$$\tau_{\contact}(-Y,K)=-\tau^*_{\contact}(Y,K):=-\tauc,$$
	\noindent and hence that $\tau_{\contact}(-Y,K)=-g(F)$.	

Recall that since $(F,K)$ is fibered,    
$$H_*({\mathcal{F}}_{\spinc_F}(\sd Y,[F],K,\sd g(F)))\cong \Z/2\Z.$$
Here, $\spinc_F$ is the $\SpinC$ structure associated to the plane field coming from the open book of $(F,K)$.  Furthermore, if we let $c_0$ be a generator, the definition of $\tau_{\contact}(-Y,K)$ (Definition \ref{defn:tau_x}) implies that $\sd I_{-g(F)}(c_0)=c(\xi)$ where, as usual, $\sd I_{-g(F)}$ is the map on homology induced by the inclusion:
$$\sd \iota_{-g(F)}: {\mathcal{F}}_{\spinc_F}(\sd Y,[F],K,\sd g(F))\hookrightarrow \CFa(\sd Y,\spinc_F).$$
\noindent On the other hand, $\sd I_{-g(F)}(c_0)=c(\xi_{(F,K)})$, by the definition of the contact invariant.

$\square$

\bigskip
\noindent {\bf Proof of Theorem \ref{thm:arbneg}:} Theorem \ref{thm:arbneg} will follow from a more precise result involving the behavior of $\tauc$ under the cabling operation, which we now review.
Recall that to a knot $(Y,K)$ and choice of Seifert surface, $F$, there is a canonical identification of the boundary of a neighborhood of $K$, $\nu(K)$ with a torus i.e. $\partial\nu(K)\cong S^1\times S^1$.  The identification is such that $\{pt\}\times S^1 \equiv \lambda$ and $S^1 \times \{pt\} \equiv \mu$, where  $\lambda$ is the longitude of $K$ coming from $F$ and $\mu$ is the meridian of $K$.  Given this identification, we can form a new knot, the $(p,q)$ {\em cable of} $K$.  By definition the $(p,q)$ {\em cable of} $K$ is the isotopy class of a simple closed curve on $\partial \nu(K)$ of slope $\frac{p}{q}$ with respect to the framing of $\partial \nu(K)$ given by $(\lambda,\mu)$.  Theorem \ref{thm:arbneg} is an immediate consequence of

 \begin{theorem}
	 \label{thm:cable}
	 Let $K\hookrightarrow Y$ be a null-homologous knot with Seifert surface $F$ and let $K_{p,q}$ denote the $(p,q)$ cable of $(Y,K)$.  Suppose $\xi$ is a contact structure on $Y$ with $\contact\ne0$.  Then for any $N,p>0$ there exists a $q=q(N,p)>0$ such that $$\tau_\xi([F],K_{p,-q})< -N,$$
 \end{theorem}

 \begin{proof} We rely on results of \cite{MyThesis} and \cite{CableII}.  In particular, \cite{CableII} proves the following 
\begin{theorem}
Let $K\hookrightarrow Y$. Pick any $l>0$. Then there exists a 
$q(l)>0$ such that for all $\sd q<\sd q(l)$ we have a chain homotopy equivalence:

$$    \Filt(Y,K_{p,\sd q},\sd g(K_{p,\sd q})+pj)\cong \Filt(Y,K,\sd g(K)+j), $$
\noindent for every $j< l$, any homology class of Seifert surface $[F]$, and every $\spinc \in \SpinC(Y)$.
\end{theorem}
\bigskip

\noindent In particular, pick $l=2g(K)+1$ and $\spinc=\spinc_\xi$ in the above theorem. Then 
$$      H_*(\mathcal{F}_{\spinc_\xi}(Y,K_{p,\sd q},\sd g(K_{p,\sd q})+2pg(K)))\cong H_*({\mathcal{F}_{\spinc_\xi}}(Y,K,g(K))). $$
However, by the adjunction inequality for knot Floer homology, the right hand side of the above is equal to $\HFa(Y,\spinc_\xi)$. Thus the inclusion 
$$\iota_{\sd g(K_{p,\sd q})+2pg(F)}:\Filt(Y,K_{p,\sd q},\sd g(K_{p,\sd q})+2pg(K))\hookrightarrow \CFa(Y)$$ is surjective on homology and hence $\tau_\xi(Y,K_{p,\sd q})\le \sd g(K_{p,\sd q})+2pg(K)$.  On the other hand 
$$g(K_{p,\sd q})=pg(K)+\frac{(p- 1)(q- 1)}{2},$$ so $$\sd g(K_{p,\sd q})+2pg(K)=pg(K)-\frac{(p-1)(q-1)}{2},$$
and thus $$\tau_\xi(Y,K_{p,\sd q})\le pg(K)-\frac{(p-1)(q-1)}{2}.$$ Now, for fixed $p$, the right hand side can be made arbitrarily negative provided we choose $\sd q<0$ to be sufficiently negative.

\end{proof}

\commentable{
\bibliographystyle{plain}
\bibliography{biblio}
}

\end{document}